\documentclass[12pt]{amsart}

\usepackage[utf8]{inputenc}
\usepackage{enumerate}
\usepackage{amsfonts,amsmath}
\usepackage{amssymb}
\usepackage{latexsym}
\usepackage{amsthm,amscd}
\usepackage{setspace} 
\usepackage{hyperref}
\usepackage{cite}
\usepackage{enumerate}
\usepackage{stmaryrd}
\usepackage{stackrel}

\newtheorem{Theorem}{Theorem}[section]
 \newtheorem{cor}[Theorem]{Corollary}
 \newtheorem{Lemma}[Theorem]{Lemma}
 \newtheorem{Proposition}[Theorem]{Proposition}
 \theoremstyle{definition}
 
 \theoremstyle{remark}
 \newtheorem{Remark}[Theorem]{Remark}
 \newtheorem{example}{Example}
 \numberwithin{equation}{section}

\def\a{{\alpha}}
\def\b{{\beta}}

\def\la{{\lambda}}

\def\M{{\mathfrak M}}
\def\Z{{\mathbb Z}}
\def\Q{{\mathbb Q}}
\def\R{{\mathbb R}}

\def\ID2{$(\text{ID}_2)$}
\def\rk{{\rm rk}}

\begin{document}

\title[]{PRINC domains and Comaximal Factorization domains}

\author{Laura Cossu}

\address{Laura Cossu \\Dipartimento di Matematica ``Tullio Levi-Civita''\\
 Via Trieste 63\\  35121 Padova, Italy}

\email{lcossu@math.unipd.it}

\author{Paolo Zanardo}

\address{Paolo Zanardo \\Dipartimento di Matematica ``Tullio Levi-Civita''\\ Via Trieste 63\\  35121 Padova, Italy}

\email{pzanardo@math.unipd.it}

 

\begin{abstract}
The notion of a PRINC domain was introduced by Salce and Zanardo (2014), motivated by the investigation of the products of idempotent matrices with entries in a commutative domain.
 An integral domain $R$ is a {\it PRINC domain} if every two-generated invertible ideal of $R$ is principal. PRINC domains are closely related to the notion of a unique comaximal factorization domain, introduced by McAdam and Swan (2004). In this article, we prove that there exist large classes of PRINC domains which are not comaximal factorization domains, using diverse kinds of constructions. We also produce PRINC domains that are neither comaximal factorization domains nor projective-free.
\end{abstract}

\maketitle

\section*{Introduction}
The notion of a PRINC domain was introduced in \cite{SalZan}, motivated by the investigation of the products of idempotent matrices with entries in a commutative domain.

The study of products of idempotent matrices over rings has raised much interest, both in the commutative and non-commutative setting, starting from a 1967 paper by J. Erdos \cite{Erdos}. An integral domain $R$ is said to satisfy {\it property ID$_n$} if any $n\times n$ singular matrix over $R$ can be written as a product of idempotent matrices. Important results on this property by Fountain \cite{Fount} and Ruitenburg \cite{Ruit} (see also \cite{Laff1} and \cite{BhasRao}) motivated Salce and Zanardo to conjecture that ``if an integral domain $R$ satisfies property ID$_2$, then it is a B\'ezout domain, i.e., an integral domain in which every finitely generated ideal is principal" (see \cite{SalZan}).
These kinds of questions have been recently investigated in \cite{CZ}, where it is proved that any domain satisfying ID$_2$ must be  a Pr\"ufer domain. The reverse implication is not true, not even for PIDs (see \cite{SalZan} and \cite{CZZ}).

As a matter of fact, PRINC domains form a large class of rings in which the above mentioned conjecture is true (see \cite{SalZan}). In Section $1$ we will give the original definition of PRINC domains; however, by our Theorem \ref{1.5} we may say, equivalently, that an integral domain is {\it PRINC} if every invertible ideal generated by two elements is principal. Besides B\'ezout domains, we may observe that unique factorization domains, local domains (i.e., domains with a unique maximal ideal) and projective-free domains (i.e., every finitely generated projective module over these domains is free) are PRINC.

After \cite{SalZan}, the study of PRINC domains was carried on by Peruginelli, Salce, and Zanardo in \cite{PerSalZan}. In Remark 1.8 of that paper, the authors remarked a close relation between PRINC domains and unique comaximal factorization domains (UCFDs for short), a notion introduced by McAdam and Swan in \cite{UCFD}. Indeed, Theorem 1.7 of \cite{UCFD} shows that a comaximal factorization domain (CFD) is a PRINC domain if and only if it is a UCFD. We refer to Section $1$ for the precise definitions.

 It was proved in \cite{PerSalZan} that a Dedekind domain is a PRINC domain if and only if it is a PID. In view of this result, it is natural to ask whether any Pr\"ufer PRINC domain is also a B\'ezout domain. In Corollary \ref{Prufer PRINC is Bezout} we answer this question affirmatively.

 Section $2$ of this paper is concerned with techniques for constructing PRINC domains. An easy but useful result shows that the union of a join-semilattice of PRINC domains is still a PRINC domain (Theorem \ref{union}).  
 In Theorem \ref{pullback_princ} we prove that, for an assigned PRINC domain $D$, the pullback $R=D+\mathfrak{M}$, where $\mathfrak{M}$ is the maximal ideal of some valuation domain containing the field of fractions of $D$, is also a PRINC domain. Theorem \ref{D[X]} supplements a result in \cite{UCFD}, by showing that a seminormal domain $D$ is PRINC if and only if the polynomial ring $D[X]$ is PRINC. 

One main purpose of the present paper is to show that the class of PRINC domains is much larger than that of unique comaximal factorization domains. In fact, in Sections 3 and 4 we construct two subclasses of PRINC domains that are not comaximal factorization domains (see Theorems \ref{monoid-domain nonCFD} and \ref{bigcup nonCFD}). We remark that the rings in the first subclass are generalizations of the monoid domains constructed by Juett \cite{Juett}. We also apply our techniques to a pair of important examples of integral domains that are not projective-free, namely the coordinate ring of the real $3$-dimensional sphere and one deep example of a UCFD with nonzero Picard group given by McAdam and Swan in Section 4 of \cite{UCFD}. We are thus able to produce classes of PRINC domains that are neither comaximal factorization domains nor projective-free (Theorem \ref{S2}).

\section{Definitions and first results}

In what follows we will always deal with commutative integral domains. For $R$ an integral domain, we denote by $R^{\times}$ its set of units. 

Two elements $a,b\in R$ are said to be an \textit{idempotent pair} if either $(a \ b)$ or $(b \ a)$ is the first row of a $2\times 2$ idempotent matrix, or equivalently if either $a(1-a)\in bR$ or $b(1-b)\in aR$ (see \cite{PerSalZan}).

It is easy to verify that any ideal generated by an idempotent pair is invertible: in fact, if say $a(1-a)\in bR$, then $(a, b) (1-a, b) = bR$.

We will say that an integral domain $R$ is a {\it PRINC domain} if every ideal generated by an idempotent pair is principal.

The next result was proved in \cite{SalZan}.

\begin{Proposition}\label{UFDprojfree_PRINC} 
If $R$ is either a UFD or a projective-free domain, then it is also a PRINC domain.
\end{Proposition}

In particular, local domains are PRINC domains, since every local domain is projective-free \cite{Kap}.

The main result on PRINC domains in \cite{SalZan} shows that these rings satisfy the conjecture mentioned in the introduction. 

\begin{Theorem}[Th. 4.6 of \cite{SalZan}]\label{princ+id2_implica_bezout}
If $R$ is a PRINC domain satisfying property \ID2, then $R$ is a B\'ezout domain.
\end{Theorem}

As a corollary, we immediately get from Proposition \ref{UFDprojfree_PRINC} that the classes of unique factorization domains and projective-free domains also satisfy the conjecture. We observe that, for $R$ projective-free, the above result was proved by Bhaskara Rao in \cite{BhasRao}, using different arguments.

\begin{example}[A UFD that is not projective-free]\label{UFD-non-proj-free}
The coordinate ring of the $3$-dimensional real sphere 
\[B_2=\frac{\mathbb{R}[X_0,X_1,X_2]}{\left( \sum_{i=0}^2X_i^2-1\right)},\]
where $X_0,X_1,X_2$ are indeterminates over the real numbers $\mathbb{R}$, is a factorial domain (cf. \cite[Prop. 8.3]{SamuelUFD}) that is not projective-free. This fact can be seen in the following way: let $F= {\bigoplus}_{i = 0}^2 B_2 e_i \cong B_2^3$, and
consider the $D$-epimorphism 
\[\phi_0:F\rightarrow B_2,\]
defined, as in \cite[Ex.~2.10]{Lam}, by $\phi_0(e_i)=X_i$, $0 \le i \le 2$. If we set $P_0={\rm Ker}\, \phi_0$, we get $F\cong P_0 \oplus B_2$, hence $P_0$, being a direct summand of a free module, is a finitely generated  projective $R$-module of rank $2$. It has been proved by different techniques (cf. \cite{Kong,Swan}) that $P_0$ is not a free module and therefore $B_2$ is not a projective-free domain.
\end{example}

\begin{example}[A PRINC domain that is  neither a UFD, nor Dedekind, nor semi-local, but is projective-free] 
In Section 4 of \cite{PerSalZan}, it is proved that the rings $\mathbb{Z}(\sqrt{-3})$ and $\mathbb{Z}(\sqrt{-7})$ are PRINC domains. So these rings are PRINC domains that are neither UFDs, Dedekind, nor semi-local, hence, by Theorem \ref{princ+id2_implica_bezout} they cannot satisfy property \ID2.
However, in \cite{PerSalZan} it is also proved that $\mathbb{Z}(\sqrt{-3})$ and $\mathbb{Z}(\sqrt{-7})$ are projective-free. This follows from the fact that their invertible ideals are principal.
\end{example}

In view of Proposition \ref{UFDprojfree_PRINC} it is natural to ask if there exist PRINC domains which are neither UFDs nor projective-free. An example of this kind of domain may be found in \cite[Sect.~4]{UCFD}. We provide a different example in Section 2.

\medskip
 
The notion of a PRINC domain is closely related to the notion of a unique comaximal factorization domain introduced by McAdam and Swan in \cite{UCFD}.

We recall the definitions. Two elements $c,d$ of an integral domain $R$ are said to be \textit{comaximal} if the ideal $( c, d ) = R$. A nonzero nonunit element $b$ of $R$ is called \textit{pseudo-irreducible} if $b \ne cd$ whenever $c,d$ are comaximal nonunits of $R$. We call $b=b_1b_2\cdots b_m$ a \textit{complete comaximal factorization} of $b$ if the $b_i$'s are pairwise comaximal pseudo-irreducible elements. We say that $R$ is a \textit{comaximal factorization domain} (CFD, for short), if every nonzero nonunit $b\in R$ has a complete comaximal factorization. If this factorization is unique (up to order and units), then $R$ is said to be a \textit{unique comaximal factorization domain} (UCFD, for short). It is worth noting that comaximal factorization domains are very common. For instance, local domains are UCFDs since every nonunit is pseudo-irreducible. Moreover, Noetherian domains and, more generally, $J$-Noetherian domains are CFDs. We refer to \cite[Lemma 1.1]{UCFD} for other examples of CFDs. 

As observed in \cite[Remark 1.8]{PerSalZan}, a CFD $R$ is a UCFD if and only if it is a PRINC domain. The crucial observation to verify this fact is that a nonzero ideal $I$ is  generated by an idempotent pair if and only if $I$ is an $S$-ideal, i.e. $I=(a,b)=(a^2,b)$, for some $a,b\in R$ (see \cite[Cor.~4.14]{Juett_12}, \cite[Prop.~1.5]{McA_S_05}). Since, by Theorem 1.7 of \cite{UCFD}, a CFD $R$ is a UCFD if and only if every $S$-ideal of $R$ is principal, the notions of PRINC domain and of UCFD coincide within the class of comaximal factorization domains. However, the notion of a PRINC domain is much more general then that of a UCFD. In Sections 3 and 4 we provide large classes of PRINC domains that are not CFDs.

\begin{example}[A CFD that is not a UCFD]
Any Dedekind domain $R$ is a comaximal factorization domain, being Noetherian. By Corollary 2.6 of \cite{PerSalZan} (or by our next Theorem \ref{1.5}) $R$ is PRINC if and only if it is a PID. Therefore any Dedekind domain that is not a PID is an example of a CFD that fails to be a UCFD. As an {\it uncommon} example, we mention the minimal Dress ring $D$ of $\R(X)$, defined as $D = \R[1/(1 + f^2): f \in \R(X)]$. This ring, whose structure is described in \cite{PruferDress}, is a Dedekind domain that is not PRINC. This can be easily seen by observing that $1/(1+X^2)$ and $X/(1+X^2)$ form an idempotent pair in $D$ but, by \cite[Prop.~2.4]{PruferDress}, $(1/(1+X^2),X/(1+X^2))$ is a $2$-generated ideal of $D$.
\end{example}

\medskip

In \cite[Cor.~2.6]{PerSalZan} it is proved that every Dedekind PRINC domain is a PID. It is natural to ask whether the analogous result for the non-Noetherian case holds, i.e., if every Pr\"ufer PRINC domain is a B\'ezout domain. This was explicitly proved to be true when $R$ has finite character (see \cite[Cor.~1.7]{PerSalZan}) and when $R$ is a CFD (see \cite[Cor.~1.9]{UCFD}). To end this section, we give a direct simple proof of the general case.

The next lemma can be proved with an intermediate step, combining Lemma 1.5 (b) in \cite{UCFD}, with Theorem 1.3 in \cite{PerSalZan}. For the reader's convenience, we give a direct proof, essentially identical to that in \cite[Lemma 1.5]{UCFD}.

\begin{Lemma}\label{invertible idempotent}
Every two-generated invertible ideal of an integral domain $R$ is isomorphic to an ideal of $R$ generated by an idempotent pair.
\end{Lemma}

\begin{proof}
Let $I=(a,b)$ be a two-generated invertible ideal of $R$. Then there exist $\lambda, \mu\in I^{-1}$ such that $\lambda a+ \mu b=1$. From this relation we derive $\lambda a (1-\lambda a)= \la b \mu a  $, hence $\lambda a, \lambda b \in R$ form an idempotent pair. Set $J=(\lambda a, \lambda b)\subseteq R$. Then $I\cong J$, with $J$ an ideal of $R$ generated by an idempotent pair.
\end{proof}

The following Theorem is an immediate consequence of Lemma \ref{invertible idempotent}.

\begin{Theorem}\label{1.5}
An integral domain $R$ is PRINC if and only if every two-generated invertible ideal of $R$ is principal.
\end{Theorem}

The following straightforward corollary answers a question left open in \cite{PerSalZan}.

\begin{cor}\label{Prufer PRINC is Bezout}
$R$ is a Pr\"ufer PRINC domain if and only if $R$ is a B\'ezout domain.
\end{cor}

\section{Constructions of PRINC domains}
In this section we provide some techniques to construct PRINC domains.

In the following proposition, we prove that PRINC domains can be easily constructed via pullbacks.

\begin{Theorem}\label{pullback_princ}
Let $D$ be a PRINC domain with field of fractions $Q$, $V$ a valuation domain containing $Q$, $\mathfrak M$ the maximal ideal of $V$. Then the pullback $R=D+ \M$ of $D$ is also a PRINC domain.
\end{Theorem}

\begin{proof}
Note that $Q \M = \M$ since $u\M=\M$ for any $u\in V^{\times}$. We recall that $u \in R^{\times}$ if and only if $u = c + z$ for some $c \in D^{\times}$ and $z \in \M$. In fact $1/c \in R$ and $z/c \in \M$, hence 
$$
{\frac{1}{c + z}} = c^{-1}\left(1 - \frac{z/c}{1 + z/c}\right) \in R. 
$$

Let us pick two elements $a, b \in R$ that form an idempotent pair, say $a(1-a)=br$ for some $r\in R$, and prove that the ideal they generate is principal. We distinguish various cases.

If  $a\notin \M$ and $b\in \M$, then $b/a\in \M$. Hence, $b\in aR$. 

If $a\in \M$, then $1-a\in R^{\times}$. Hence $a=br(1-a)^{-1}\in bR$. 

Finally, assume that $a,b\notin \M$. Say $a \equiv a'$, $b \equiv b'$, $r \equiv r'$ modulo $\M$, where $a', b', r' \in D$ (note that $r' = 0$ if $a - 1 \in \M$). From the condition $a(1-a)=br$, we get $a'(1 - a') \equiv b'r'$ modulo $\M$, and therefore $a'(1 - a') = b'r'$.  Then $a'$ and $b'$ form an idempotent pair in $D$, thus, being $D$ a PRINC domain, there exists an element $d\in D$ such that $a'D+b'D=dD$. Finally, since $a$, $b$ are associated in $R$ to $a'$, $b'$, respectively, we get $aR + bR = a'R + b'R =dR$.
\end{proof}

We observe that an analogous result for UCFDs can be found in Section 3 of \cite{UCFD}.

The preceding result allows us to give an example of a PRINC domain that is neither factorial nor projective-free, namely, a pullback of the coordinate ring $B_2$ of the real $2$-sphere $S^2$, as in Example \ref{UFD-non-proj-free}. In the last section, we will also consider the deep example given by McAdam and Swan in \cite{UCFD}.

Recall that the \textit{rank} of a torsion-free module $X$ over an integral domain $R$, denoted by $\rk_R(X)$, is the rank of a maximal free submodule of $X$, i.e., the maximal number of elements of $X$ that are linearly independent over $R$ (see for instance \cite[Ch.I, p.~24]{FuSal_nn}). The rank coincides with the dimension of the $Q$-vector space $Q \otimes_R X$, where $Q$ is the field of fractions of $R$. 

For any finitely generated torsion-free $R$-module $X$, we denote by $g_R(X)$ the minimum number of generators needed to generate $X$. It is well-known that $g_R(X) \ge \rk_R(X)$, where equality implies that $X$ is a free $R$-module. Indeed, if $X=\langle x_1,\dots, x_n\rangle$ ($n$ minimum), then also $Q \otimes_R X$ is generated by $x_1, \dots, x_n$, hence, $n\geq \dim (Q \otimes_R X)=\rk_R (X)$.  Moreover, if $n=\rk_R (X)$, then $x_1, \dots, x_n$ must be linearly independent, hence $X=\bigoplus_{i=1}^n R x_i$.

We will need the next lemma, concerned with extensions of non-projective free integral domains.

\begin{Lemma} \label{Lemma}
Let  $R$ be an integral domain of the form $R = D + M$, where $D$ is a subring of $R$ and $M$ is an ideal of $R$ such that $D \cap M = 0$. If $D$ is not projective-free, then also $R$ is not projective-free.
\end{Lemma}

\begin{proof}
Since $D$ is not projective-free, there exists $k \ge 2$ such that $D^k =  \bigoplus_{j = 1}^k D e_j =P_1 \oplus P_2$, where, say, $P_1$ is not free, hence ${\rk}_D(P_1) < g_D(P_1)$. Let $\pi_i : D^k \to P_i$ be the canonical projection ($i = 1,2$), and define $e_{ij} = \pi_i(e_j)$, for $1 \le j \le k$. 

Let us consider the free $R$-module $R^k =  \bigoplus_{j = 1}^k R e_j \supset D^k$ and its $R$-submodules $Q_i = \sum_{j = 1}^k R e_{ij}$ ($i = 1, 2$). We extend by $R$-linearity the assignments $e_j \mapsto e_{ij}$ ($1 \le j \le k$), obtaining two surjective maps $\b_i : R^k \to Q_i$. Since $\beta_1(e_{ij}) = \pi_1(e_{ij})$ (and analogously for $\beta_2$), we see that $\b_1$ is the identity and $\b_2$ is zero when restricted to $Q_1$ (and symmetrically for $Q_2$).  Since $\{ e_1, \dots, e_k \} \subset P_1 + P_2 \subset Q_1 + Q_2$, we conclude that  $R^k = Q_1 + Q_2$. Moreover, if $z \in Q_1 \cap Q_2$, we get $z = \b_{1}(z) = \b_2(z) = 0$. It follows that $R^k = Q_1 \oplus Q_2$.

Let us observe that $D \cap M = 0$ easily yields $\sum_j D e_j \cap \sum_j M e_j = 0$. Then we also get $\sum_j D e_{1j} \cap \sum_j M e_{j} = 0$, since $e_{1j} \in D^k$ for every $j \le k$.
Let us verify that $g_D(P_1) \le g_R(Q_1)$.
 Indeed, let $z_1, \dots, z_m$ be a set of generators of $Q_1$, with $m = g_R(Q_1)$. Say $z_n = \sum_j r_{jn} e_{j}$, where $r_{jn} = d_{jn} + t_{jn} \in R$, $d_{jn} \in D$, $t_{jn} \in M$. Then
 $\sum_j d_{jn} e_j + \sum_j t_{jn} e_j \in \sum_j D e_{1j} + \sum_j M e_{1j}$ yields 
  $x_n = \sum_j d_{jn} e_j \in \sum_j D e_{1j} = P_1$, for every $n > 0$. Symmetrically, every $e_{1j}$ lies in  $\sum_n D x_n$. We conclude that $\sum_n Dx_n = P_1$, hence $g_D(P_1) \le g_R(Q_1)$. The same argument shows that $g_D(P_2) \le g_R(Q_2)$.

 Now we assume, for a contradiction, that $R$ is projective-free, so $Q_1$, $Q_2$ are free. Since $P_1$ is not free, we get $\rk_D(P_1) < g_D(P_1) \le g_R(Q_1) = \rk_R(Q_1)$, so $\rk_D(P_1) + \rk_D(P_2) = \rk_D(D^k) = \rk_R(R^k) = \rk_R(Q_1) + \rk_R(Q_2)$ yields $\rk_D(P_2) > \rk_R(Q_2) = g_R(Q_2) \ge g_D(P_2)$, impossible. We reached a contradiction.
\end{proof}

We are now in the position to provide a rather general example of a PRINC domain which is neither a UFD nor projective-free.

\begin{Proposition}\label{PRINC_nonUFD_nonPROJ}
Let $B_2$ be the coordinate ring of the $3$-dimensional real sphere, $Q$ its field of fractions, $V$ a valuation domain containing $Q$, $\M$ the maximal ideal of $V$. Then the pullback $R = B_2 + \M$ is a PRINC domain that is neither a UFD nor projective-free.
\end{Proposition}

\begin{proof}
The coordinate ring $B_2$ is a PRINC domain since it is a UFD, thus it follows from Theorem \ref{pullback_princ} that $R$ is a PRINC domain. Obviously, $R$ is not a UFD: take any nonzero nonunit $d \in B_2$ and $z\neq 0 \in \M$. Then $z/d^n \in R$ for every $n > 0$, hence $z$ cannot be uniquely written as a finite product of irreducible elements. Finally $R$ is not projective-free. In fact $B_2$ is not projective-free and $B_2 \cap \M = 0$, so we are in the position to apply Lemma \ref{Lemma}.
\end{proof}

\medskip

The next easy result provides a useful way to construct PRINC domains.

\begin{Theorem} \label{union}
Let the integral domain $R$ be the union of a join-semilattice of PRINC domains. Then $R$ is a PRINC domain.
\end{Theorem}

\begin{proof}
Say $R = \bigcup_{\a \in \Lambda}R_\a$, where the $R_\a$ are PRINC domains, and $\{R_\a: \a \in \Lambda \}$ is a join-semilattice (with respect to inclusion). Let $a,b$ be an idempotent pair on $R$. Then, there exists $\lambda\in\Lambda$ such that $a,b$ is an idempotent pair on $R_{\lambda}$. Since $R_{\lambda}$ is a PRINC domain, $(a,b)R_{\lambda}$ is a principal ideal of $R_{\lambda}$, hence $(a,b)R$ is a principal ideal of $R$. By the definition, it follows that $R$ is a PRINC domain.
\end{proof}

\medskip

According to Swan \cite{Swan_seminormal} an integral domain $D$ is said to be {\it seminormal} if whenever $a,b\in D$ satisfy $a^2=b^3$, then there is an element $c\in D$ such that $c^3=a$ and $c^2=b$. It is well known from \cite[Th.~1.6]{Gilmer_seminormal} and \cite[Th.~1]{BrewerCosta} that the following conditions are equivalent for an integral domain $D$:
\begin{enumerate}
\item $D$ is seminormal;
\item ${\rm Pic}\,(D)= {\rm Pic}\,(D[X])= {\rm Pic}\,(D[X_1,\dots,X_n])$ for all $n\geq 1$;
\item if $\alpha\in Frac\,(D)$ and $\alpha^2,\alpha^3 \in D$, then $\alpha\in D$
\end{enumerate}

 The next theorem somehow supplements an analogous important result, namely Theorem 2.3 of \cite{UCFD}, that was proved for UCFDs. As a matter of fact, we will later prove the existence of seminormal PRINC domains that are not UCFDs, so the following is not a consequence of \cite[Th.~2.3]{UCFD}, although the proof is based on the arguments given in \cite{UCFD}.

\begin{Theorem} \label{D[X]}
Let $D$ be an integral domain, $\{X_{\lambda}\}_{\lambda\in\Lambda}$ a nonempty set of indeterminates over $D$. Then $D[\{X_{\lambda}\}_{\lambda\in\Lambda}]$ is a PRINC domain if and only if $D$ is a seminormal PRINC domain.
\end{Theorem}

\begin{proof} By Theorem \ref{union} we may assume that $\Lambda$ is finite. From the characterizations of seminormal domains recalled above, $D$ is a seminormal domain if and only if $D[X]$ is seminormal. Therefore it is not restrictive to assume in what follows that $n = 1$.

Let us assume that $D[X]$ is a PRINC domain. To show that $D$ is a PRINC domain, let $a,b\in D$ be an idempotent pair, say $a(1-a)\in bD\subseteq b D[X]$. Then $a D[X]+ bD[X]=p D[X]$ for some $p\in D[X]$ and $a D+b D=p(0) D$, as desired. It remains to prove that $D$ is seminormal. Let $Q$ be the field of fractions of $D$. Assume, for a contradiction, that $D$ is not seminormal. Then there exists $\alpha\in Q \setminus D$ such that $\alpha^2$ and $\alpha^3$ are elements of $D$. Set $a=1-\alpha X$ and $b=1+\alpha X$. Then $ab=1-\alpha^2 X^2$ and $\alpha^2$ are comaximal elements of $D[X]$. In fact 
\begin{equation}\label{comaximal}
(1+\alpha^2 X^2)(1-\alpha^2 X^2)+\alpha^4 X^4=1.
\end{equation} 
One can directly verify that \[(1+\alpha^2 X^2)ab(1-(1+\alpha^2 X^2)ab)=\alpha^2 b(-\alpha^5X^7+\alpha^4 X^6-\alpha^3 X^5+\alpha^2X^4),\] so $(1+\alpha^2 X^2)ab$ and $\alpha^2 b$ form an idempotent pair in $D[X]$. Let $I=(1+\alpha^2 X^2)ab D[X]+\alpha^2 b D[X]$ the ideal of $D[X]$ generated by these elements. We claim that $I$ cannot be a principal ideal. In fact, if there exists $f\in D[X]$ such that $I=f D[X]$, then $f Q[X] = b \,((1+\alpha^2 X^2) \,a,\alpha^2)Q[X]$. The equality (\ref{comaximal}) implies that  $((1+\alpha^2 X^2)	\,a,\alpha^2)= Q[X]$, therefore $f Q[X]=b Q[X]$. So $f=u (1 + \a X)$ for a suitable $u\in Q$. Since $f \in D[X]$, both $u$ and $u\alpha$ are elements of $D$. Moreover, since $(1+\alpha^2 X^2)ab$ is a multiple of $f$ in $D[X]$, it follows that $u\in D^{\times}$ and $\alpha\in D$, impossible. Therefore $I$ is a non-principal ideal generated by an idempotent pair, hence $D[X]$ is not a PRINC domain, a contradiction.

Conversely, let $D$ be a seminormal PRINC domain. Then the canonical map of ${\rm Pic}\,(D)$ into $ {\rm Pic}\,(D[X])$ is an isomorphism. Let $J$ be a two-generated invertible ideal of $D[X]$. Then there exists an invertible ideal $I$ of $D$ such that $I\, D[X] \cong J$. Then $I\, D[X]=(f_1(X),f_2(X))$ and, since $I=I \,D[X]\cap D$, $I$ is generated by the two constant terms $f_1(0),f_2(0)\in D$.  So $D$ PRINC yields $I$ principal, by Theorem \ref{1.5}, and, consequently, $J$ is principal, as well. We conclude that $D[X]$ is a PRINC domain, since $J$ was arbitrary.
\end{proof} 

\begin{example}
Consider the following local domain $D = K[y^2, y^3]_{(y^2,\,  y^3)}$ ($K$ a field, $y$ an indeterminate). Being local, $D$ is a PRINC domain. Moreover $D$ is not seminormal, since $y \notin D$. By Theorem \ref{D[X]} we conclude that $D[X]$ {\it is not} a PRINC domain.
\end{example}

The next examples are inspired by those given by Juett in \cite{Juett}.

In what follows $S$ will denote a commutative cancellative torsion-free semigroup.  We adopt the usual notation $D[X; S]$ for the semigroup ring of $S$ over the integral domain $D$, and $X^s$, $s \in S$, for the formal powers over $S$ that satisfy the usual relations (see \cite[Ch.~2]{Gilmer_comm_semigroup}). Under our assumptions on $S$, $D[X; S]$ is an integral domain (see \cite[Th.~8.1]{Gilmer_comm_semigroup}).  Following \cite{ACHZ} we say that a commutative semigroup $S$ is {\it locally cyclic} if, for any assigned $s_1,s_2 \in S$, there exists $s\in S$ such that $s_1 \mathbb N + s_2 \mathbb N \subseteq s \mathbb N$. If $S$ is locally cyclic, it is easy to show that $D[X; S]$ is the union of the join-semilattice of its polynomial subrings $D[X^s]$, $s \in S$.

\medskip

Now we consider a different kind of construction.
Let $D_0$ be an integral domain, $D = D_0[X_n : n > 0]$, take the $D$-ideal 
$$
J = (X_i - X_{i +1} (1 + X_{i+1}): i > 0),
$$
and define $R = D/J$. 

We use the notation $x_n = X_n + J$, so $R = D_0[x_n : n > 0]$. For each $n > 0$, let $R_n = D_0[x_1, \dots, x_n]$. The equalities $x_i = x_{i+1}(1 + x_{i+1})$ imply $R_n = D_0[x_n]$. By the definition we get $R_n \subseteq R_{n+1}$, hence $R = \bigcup_{n > 0} R_n$.

Moreover, $R$ is an integral domain and $D_0[x_n]$ is a polynomial ring for every $n > 0$. These rather expectable but not quite evident facts are proved in the next proposition.

\begin{Proposition}
In the above notation, $R_n = D_0[x_n] \cong D_0[Z]$, where $Z$ is an indeterminate. In particular, $R = \bigcup_{n > 0} R_n$ is an integral domain. 
\end{Proposition}

\begin{proof}

It suffices to verify that $x_n$ is transcendental over $K = {\rm Frac}(D_0)$. So, by extending the scalars, we may  assume $D_0 = K$ and $D = K[X_i : i > 0]$. Let $F(x_n) = x_n^k(a_0 + a_1x_n + \dots + a_hx_n^h) \in K[x_n]$, where $a_0 \ne 0$. We have to show that $F(x_n) \ne 0$, or, equivalently $F(X_n) \notin J $. 

Let us consider the localization $T = D_{\M}$, where $\M = (X_i : i > 0)$. Note that $J \subseteq \M$ implies that $JT$ is a proper ideal of $T$. Then $u_i = 1 + X_i$ is a unit of $T$ for every $i \ge 2$; let $Z_i = X_{i-1} - u_iX_i \in JT$. Since the $u_i$ are units of $T$, it is readily seen that $(X_1, Z_2, \dots, Z_m)T = (X_1, X_2, \dots, X_m)T$, for every $m \ge 2$. It follows that \[T = K[X_1, Z_2, \dots, Z_m, X_i : i \ge m+1]_\M,\text{ for every } m \ge 2.\] Take the ideal $J_m = (Z_2, \dots, Z_m)T$. It is then clear that $T/J_mT \cong K[X_1, X_i : i \ge m+1]_{\M_m}$, where $\M_m = (X_1, X_i : i \ge m +1)$. It follows that $J_mT$ is a prime ideal of $T$, hence $JT = \bigcup_{m \ge 2} J_m T$ is also a prime ideal of $T$.

To reach the desired conclusion, it suffices to show that $F(X_n) \notin JT \supset J$. Since $a_0 + a_1X_n + \dots + a_hX_n^h$ is a unit of $T$, it is enough to verify that $X_n^k \notin JT$. Assume, for a contradiction, that $X_n^k \in JT$. Since $X_{i-1} \equiv X_{i}u_i$ modulo $JT$, it follows that $X_i^k \in JT$ for every $i > 0$. In particular, $X_1^k \in JT$, hence $X_1 \in JT$, since $JT$ is a prime ideal. Then we get
$$
X_1 = \sum_{i = 2}^{s} (X_{i-1} - u_iX_{i})g_i = \sum_{i = 2}^{s} Z_i g_i,
$$
for suitable $g_i \in T$. As observed above, we have $(X_1, Z_2, \dots, Z_s)T = (X_1, X_2, \dots, X_s)T$. By Nakayama Lemma, the $X_1 + \M, Z_2 + \M, \dots, Z_s + \M$ are linearly independent in the $T/\M$-vector space $\M/\M^2$. But the above equality yields $X_1 \in (Z_2, \dots, Z_s)T$, impossible.
\end{proof}

\begin{Theorem} \label{D[X; S] PRINC}
Let $D$ be a seminormal PRINC domain. Then 

(i) if $S$ is a locally cyclic semigroup, then the semigroup domain $D[X; S]$ is a PRINC domain;

(ii) if $R = \bigcup_{n > 0} R_n$, where $R_n = D[x_n]$ and $x_n = x_{n+1}(1 + x_{n+1})$ for every $n > 0$, then $R$ is a PRINC domain.
\end{Theorem}

\begin{proof}
Both $D[X; S]$ and $R$ are unions of a join-semilattice of polynomial rings isomorphic to $D[Z]$, which is a PRINC domain by Theorem \ref{D[X]}. Thus the results follow from Theorem \ref{union}.
\end{proof}

\section{Constructions of non-CFD PRINC domains}

In this section we provide some constructions that produce PRINC domains that are not CFD domains. So the class of UCFD domains is strictly contained in the class of PRINC domains.
\medskip

According to \cite{ACHZ}, we say that a (commutative cancellative torsion-free) monoid $S$ is {\it pure} if it is order-isomorphic to a submonoid of $(\Q^+,+)$, it is locally cyclic, and for each $s\in S$ there exists a natural number $n>1$, depending on $s$, with $s/n\in S$. Typical examples for $S$ are $\Z_{T}\cap \Q^+$, where $T\neq\{1\}$ is any multiplicatively closed subset of $\Z\cap \Q^+$.
If for each $s\in S$, there is a natural number $n>1$ such that $s/n\in S$ is never a power of $p$, then we call $S$ a {\it $p$-pure} monoid (cf. \cite{Juett}). Clearly, a pure monoid is a $0$-pure monoid.

\begin{Theorem}\label{monoid-domain nonCFD}
Let $D$ be a seminormal PRINC domain of characteristic $\chi$ and $S$ a pure monoid such that:
\begin{enumerate}
\item if $\chi=0$ and $S$ is $p$-divisible for some prime number $p>1$, then $D\supseteq \Z[1/p]$;
\item if $\chi=0$ and $S$ is not $p$-divisible for every $p$, then $D\supseteq \Q$;
\item if $\chi\neq 0$, then $S$ is $\chi$-pure.
\end{enumerate} 
In all the above cases, the monoid domain $D[X;S]$ is a seminormal PRINC domain that is not a CFD.
\end{Theorem}

\begin{proof} 
In view of Theorem \ref{D[X; S] PRINC} we already know that $D[X; S]$ is a PRINC domain. We can readily verify that it is also seminormal, using the fact that $D[X; S]$ is the union of a join-semilattice of subrings isomorphic to $D[Z]$, which is seminormal by Theorem \ref{D[X]}. Assume, for a contradiction, that $D[X; S]$ is a CFD. Then it is also a UCFD, by Theorem 1.7 of \cite{UCFD} (see the first section of the present paper). Let $s$ be any nonzero element of $S$. Since $S$ is pure, then $s = nt$, with $t \in S$ and $n > 1$. In case (1) we may take $n=p$, in case (3) $n$ is coprime with $\chi$ since $S$ is $\chi$-pure. We get $1 - X^s = (1-X^t)(1+\sum_{i=1}^{n-1}(X^t)^i)$. Let us show that $1-X^t$ and $1+\sum_{i=1}^{n-1}(X^t)^i$ are comaximal nonunits of $D[X;S]$ and therefore $1 - X^s$ is not pseudo-irreducible. Set $Z=X^t$. Dividing the polynomial $1+\sum_{i=1}^{n-1}Z^i$ by $1-Z$ we get $1+\sum_{i=1}^{n-1}Z^i=(1-Z) \,q(Z)+ n$, where $n$ is a unit of $D$: in case (1) since $n=p$ and $1/p\in D$, in case (2) since $D\supseteq \Q$, in case (3) since $D\supseteq \mathbb{F}_{\chi}$ and $n$ is coprime with $\chi$. Similarly, $1-X^t$ has a nontrivial comaximal factorization of this form and its pseudo-irreducible factors are clearly comaximal with $1+\sum_{i=1}^{n-1}(X^t)^i$. In this way we can construct arbitrary long comaximal factorizations of $1-X^s$ contradicting the fact that $D[X;S]$ is a UCFD.

\end{proof}

\medskip

In the notation of the preceding section, let $D_0$ be a seminormal PRINC domain, $D = D_0[X_n : n > 0]$, $J = (X_i - X_{i +1} (1 + X_{i+1}): i > 0)\subseteq D$, and consider the ring $R = D/J = \bigcup_{n > 0}R_n$, where $R_n=D[x_n]$. We want to show that $R$ is not a CFD.  

\begin{Theorem} \label{bigcup nonCFD}
 Let $R=  \bigcup_{n > 0}R_n$ be as above. Then $R$ is a seminormal PRINC domain that is not a CFD.
\end{Theorem}

\begin{proof}
By Theorem \ref{D[X; S] PRINC}, we already know that $R$ is a PRINC domain. Like in the proof of Theorem \ref{monoid-domain nonCFD}, we see that it is also seminormal.  Let us assume, for a contradiction, that $R$ is a CFD. Then $R$ is actually a UCFD, being PRINC and CFD, again by Theorem 1.7 of \cite{UCFD}. Assume that $x_1$ is a product of $n$ pairwise comaximal pseudo-irreducible elements. Since $R$ is a UCFD $n$ does not depend on the complete comaximal factorization. Take $m>n$. Then,
\begin{equation}\label{xn}
x_1 = x_m \prod_{i = 2}^m(1 + x_i),
\end{equation}
where the factors on the right side are comaximal. In fact $(x_m,1 + x_i)=R$ for $i\leq m$ since $x_i\in (x_m)$.  For $i<j$, we have $(x_i)\subseteq (x_{j-1})\subseteq (1+x_j)$, so $(1+x_i, 1+x_j)=R$.
So $x_1$ is a product of $m$ pairwise comaximal elements and each element has a complete comaximal factorization. It follows that $x_1$ has a complete comaximal factorization with $k\geq m>n$ elements, a contradiction.
\end{proof}

We remark that Theorems \ref{monoid-domain nonCFD} and \ref{bigcup nonCFD} provide examples of PRINC domains to which our Theorem \ref{D[X]} applies, but Theorem 2.3 of \cite{UCFD} is not applicable, since they are not CFDs. 

\begin{Remark}
In the notation of Theorems \ref{monoid-domain nonCFD} and \ref{bigcup nonCFD}, if $D = K$ is a field, then the corresponding domains are examples of non-CFD B\'ezout domains. Indeed, it is readily seen that within a union of a join-semilattice of subrings that are PID (isomorphic to $K[Z]$ in the present cases) the finitely generated ideals are principal. Note that, since B\'ezout domains are projective-free, all these rings are also examples of non-CFD projective-free domains.
\end{Remark}

We end this section constructing integral domains without pseudo-irreducible elements. An alternative proof of the following proposition \ref{Juett} can be obtained properly applying \cite[Th.~3.1]{Juett}. 

\begin{Proposition} \label{Juett}
Let $K$ be an algebraically closed field of characteristic $\chi$, $\Gamma$ a subgroup of $\Q$ divisible by a prime $p \ne \chi$. Then the group domain $R = K[X; \Gamma]$ does not contain pseudo-irreducible elements.
\end{Proposition}

\begin{proof}
We first note that $X^s$ is a unit of $R$ for every $s \in \Gamma$, since $\Gamma$ is a group. Let $\varphi$ be any nonzero nonunit element of $R$. Since $\Gamma$ is a locally cyclic abelian group, $\varphi \in K[X^t, X^{-t}]$ for a suitable $ t \in \Gamma^+$. So we may write $\varphi=f/X^{nt}$, where $f\in K[X^t]$, $f(0)\neq 0$ and $n\in\Z$ is suitably chosen. Since $K$ is algebraically closed, we get $f = a \prod_{i = 1}^m (X^t - a_i)$, for suitable $a, a_i \in K$, where $a_i \ne 0$ for all $i \le m$. Then $f$ is not pseudo-irreducible whenever $a_i \ne a_j$ for some $i \ne j$. So we may reduce ourselves to the case $f = (X^t - b)^m$, $b \in K^{\times}$. Since $\Gamma$ is $p$-divisible, we get $X^{t} - b = b(Z^p -1) =  b(Z-1) ( 1 + \sum_{i = 0}^{p-1} Z^i)$, where $Z = X^{t/p}b^{-1/p} \in R$. The elements $Z-1$ and $1 + \sum_{i = 0}^{p-1} Z^i$ are comaximal, since $p \ne \chi$, hence we readily conclude that $f$ is not pseudo-irreducible.
\end{proof}

\section{PRINC domains neither CFD nor projective-free }

The purpose of this last section is to provide classes of PRINC domains that are neither CFD nor projective-free.

Our next constructions are based on a beautiful example due to McAdam and Swan \cite{UCFD}. We briefly recall it, following their notations.

Let $B_n = \R[X_0, \dots, X_n]/(\sum_{i = 0}^n X_i^2 - 1)$ be the ring of the real-valued polynomials on the $n$-sphere $S^n$, $n \ge 2$. Let $A_n$ be the subring of $B_n$ of the even functions, namely $A_n = \R[x_ix_j : 0 \le i, j \le n]$, where $x_i = X_i + (\sum_{i = 0}^n X_i^2 - 1)$. Then the following hold (see \cite{UCFD}):

(i) ${\rm Pic}(A_n) \cong \Z/2\Z$;

(ii) $n + 1$ is minimal number of generators of the prime ideal $P_n = (x_0x_i : 0 \le i \le n)$.

From (i), (ii) it follows that the class of $P_n$ generates ${\rm Pic}(A_n)$, and that any two-generated ideal of $A_n$ is principal, hence $A_n$ is a PRINC domain. Since $A_n$ is Noetherian, it is also a CFD, hence $A_n$ is a UCFD, by Theorem 1.7 of \cite{UCFD}. Of course $A_n$ is not projective-free since ${\rm Pic}(A_n)\neq 0$.

The next fact was not remarked in \cite{UCFD}.

\begin{Lemma} \label{A_n seminormal}
In the above notation, $A_n$ is a seminormal ring.
\end{Lemma}

\begin{proof}
As observed in \cite{UCFD}, we have an isomorphism 
$$
A' = A_n[x_0^{-2}] \cong \R[Y_1, \dots, Y_n][(1 + \sum Y_i^2)^{-1}],
$$
hence $A' $ is integrally closed. Let us pick $f, g \in A_n \subset A'$  such that $f^3 = g^2$. Since $A'$ is seminormal, there exists $\a \in A'$ such that $\a^2 = f$ and $\a^3 = g$. Say $\a = c x_0^{-2m}$, where $c \in A_n$, $m \ge 0$. Assume, for a contradiction, that $m \ge 1$ and $c \notin x_0^2 A_n$. Then $\a^2 \in A_n$ yields $c^2 \in x_0^{2m} A_n \subset P_n$, hence $c \in P_n$, say $c = x_0 \sum \b_i x_i$, where $\b_i \in A_n$. Moreover $(c/x_0)^2 \in x_0^{4m - 2}A_n \subset A'$ yields $c/x_0 \in A'$, since $A'$ is integrally closed. So $c = x_0 h$, where $h \in A'$, which is impossible, since both $c$ and $h$ are even functions. We conclude that $m = 0$, so that $\a \in A_n$, and $A_n$ seminormal follows.
\end{proof}

The above discussions furnish two basic examples of seminormal PRINC domains that are not projective-free, namely $A_n$ and $B_2$ ($B_2$ is seminormal since it is a UFD, as recalled in Example \ref{UFD-non-proj-free}).

In our final result we construct PRINC domains that are neither CFD nor projective-free. Note that it is applicable to both $A_n$ and $B_2$ since $A_n, B_2\supseteq \Q$.

\begin{Theorem}\label{S2}
Let $D$ be a seminormal PRINC domain non-projective-free. 

(i) Let $R=D[X;S]$ where $D$ and $S$ satisfy the hypotheses of Theorem \ref{monoid-domain nonCFD}.

(ii) Let $R = \bigcup_{n > 0} R_n$, where $R_n = D[x_n]$ and $x_n = x_{n+1}(1 + x_{n+1})$, for $n > 0$. 

In both the above cases, $R$ is a PRINC domain that is neither a CFD nor a projective-free domain. 
\end{Theorem} 

\begin{proof}
(i) Since $D$ is a seminormal PRINC domain, it follows from Theorem \ref{D[X; S] PRINC} that $R$ is PRINC. Moreover, Theorem \ref{monoid-domain nonCFD} implies that it is not a comaximal factorization domain. It remains to prove that it is not projective-free. We observe that $R = D + M$, where $M$ is the $R$-ideal generated by the $X^s$, $s \in S$, hence we can apply Lemma \ref{Lemma}.
  
  (ii) The argument is similar to the previous one. Here $R$ is not a CFD by Theorem \ref{bigcup nonCFD} and $ R= D + M$, where $M$ is the ideal generated by the $x_n$'s.

\end{proof}

\begin{Remark}
The union of a join-semilattice of seminormal integral domains is clearly seminormal, as well: just apply the equivalent condition (3) for seminormality. So the rings $R$ constructed in Theorem \ref{S2} are seminormal, since they are unions of seminormal domains isomorphic $D[Z]$. Then, starting with these new rings, we can iterate and combine the constructions made in Theorem \ref{S2}, producing large classes of non-CFD PRINC domains, that are not projective-free, by Lemma \ref{Lemma}.
\end{Remark}

\section*{Acknowledgments}
This research was supported by Dipartimento di Matematica ``Tullio Levi-Civita'' Università di Padova, under Progetto BIRD 2017 - Assegni SID (ex Junior) and Progetto DOR1690814 ``Anelli e categorie di moduli''. The first author is a member of the Gruppo Nazionale per le Strutture Algebriche, Geometriche e le loro Applicazioni (GNSAGA) of the Istituto Nazionale di Alta Matematica (INdAM). The authors thank the referee for several useful suggestions and comments.

\bibliographystyle{plain}

\end{document}